\documentclass[letterpaper,10pt, conference]{ieeeconf}

\IEEEoverridecommandlockouts                           
\usepackage{graphicx}
\usepackage{amssymb}
\usepackage{epstopdf}
\usepackage{amsmath}
\usepackage{multirow}
\usepackage{color}
\usepackage{caption}
\usepackage{subfigure}
\usepackage[sort]{cite}

\usepackage[linesnumbered,boxed]{algorithm2e}

\newcommand{\zsmargin}[2]{{\color{blue}#1}\marginpar{\color{blue}\raggedright\footnotesize [ZS]:
#2}}

\def\Pr{\mathop{\rm Pr}\nolimits} 

\usepackage{amsfonts}
\newtheorem{theorem}{Theorem}
\newtheorem{lemma}{Lemma}
\newtheorem{assumption}[theorem]{Assumption}

   \renewcommand{\baselinestretch}{0.935}  
   
\title{\LARGE \bf Risk Aversion in  Finite Markov Decision Processes \\Using Total Cost Criteria
and Average Value at Risk}
\author{  
Stefano Carpin \and Yin-Lam Chow \and Marco Pavone \thanks{
 Y.-L. Chow and M. Pavone
are with Stanford University, Stanford, CA, USA.
 S. Carpin is with the School of Engineering, University of California-Merced, CA, USA.
\indent S. Carpin is partially supported by the Army Research Lab under contract  MAST-CNC-15-4-4. Y-L. Chow is partially supported by The Croucher Foundation doctoral scholarship.
M. Pavone is partially supported by the Office of Naval Research, Science of Autonomy Program, under Contract N00014-15-1-2673.
 Any opinions, findings, and conclusions or recommendations expressed in these materials are those of
 the authors and should not be interpreted as representing the official policies, either expressly or
 implied, of the funding agencies of the U.S. Government.  \newline
}
}

\begin{document}
\maketitle \thispagestyle{empty} \pagestyle{empty}
\begin{abstract} 
In this paper we present an algorithm to compute risk averse
policies in Markov Decision Processes (MDP) when the total cost 
criterion is used together with the average value at risk (AVaR) metric.
Risk averse policies are needed when large deviations from the 
expected behavior may have detrimental effects, and conventional MDP algorithms
usually ignore this aspect.
We provide conditions for the structure of the underlying 
MDP ensuring that approximations for the exact problem 
can be derived and solved efficiently.
Our findings are novel inasmuch as average value at risk 
has not previously been considered in association with the total cost criterion.
Our method is demonstrated in a rapid deployment scenario,
whereby a robot is tasked with the objective of reaching a target
location within a temporal deadline where increased speed is associated
with increased probability of failure. We demonstrate that
the proposed algorithm not only produces a risk averse policy reducing
the probability of exceeding
the expected temporal deadline, but also provides the statistical distribution of costs,
thus offering a valuable analysis  tool.
\end{abstract}

\section{Introduction}
Markov Decision Processes (MDPs) are extensively used to solve 
sequential stochastic decision making problems in robotics 
\cite{ProbabilisticRobotics} and other disciplines \cite{HanbookMDP}. 
A solution to an MDP problem instance provides a policy mapping
states into actions with the property of optimizing  (e.g., minimizing) 
in expectation a given objective function.
 In many practical situations a formulation
based on expectation only is, however, not sufficient. 
This is the case, in particular, when variability in the system's behavior can cause a highly undesirable outcome.
For example, in an autonomous navigation system,
a robot attempting to minimize the expected length of the 
traveled path will likely travel close to obstacles, and
 a large deviation from the planned
path may result in a collision causing a huge loss (e.g., damage to an expensive robot
or failure of the whole mission altogether).
Metrics that study deviations from the expected value are often referred to as {\em risk metrics}. 

Typical
solution algorithms for MDP problems, like value iteration or policy iteration,
aim exclusively at optimizing the expected cost. The computed policies are therefore labeled as  {\em risk neutral}.
The problem of quantifying risk associated with random variables has a rich history
and is often related to the problem of managing financial assets \cite{artzner1999coherent}. In fact, many risk-related studies
motivated by financial problems 
have recently found applications in domains such as robotics \cite{serraino2013conditional}. The term {\em risk aversion}
refers to a preference for stochastic realizations with limited deviation from the 
expected value. In risk averse optimal control one may prefer a policy with higher cost in expectation 
but lower deviations to one with lower cost but possibly larger deviations.
Particularly in the context of robotic planning, introducing risk aversion in MDPs is crucial to guarantee mission safety. However introducing risk aversion in MDPs creates a number of additional theoretical and computational hurdles. For example, in risk averse
MDPs, optimal policies are not guaranteed to be Markov stationary but are instead history dependent.

Average Value at Risk (AVaR -- also known as Conditional Value at Risk or CVaR) is a  risk metric that has gained notable popularity in the area of risk averse control \cite{artzner1999coherent}, \cite{RockafellarCVAR}. For a given 
random value and a predetermined confidence level, the AVaR  is the \emph{tail average} 
of the distribution exceeding a given confidence level (see Section \ref{sec:preliminaries}
for a formal definition).
Risk averse policies considering the AVaR metric have been studied for the case of MDPs with finite horizon
and discounted infinite horizon cost criteria.
In this paper we instead consider how  the AVaR  metric can be applied 
when an undiscounted, total cost criterion is considered. In fact, such cost criterion appears particularly useful and natural for robotic applications, whereby one is usually interested in optimizing the {\em undiscounted, total cost} accrued during a mission until a {\em random}, mission-dependent stopping time. (As an aside, the total cost criterion is the typical cost model for stochastic shortest path problems, see, e.g., \cite{DPB-JNT:91,HY-DPB:12}.)
The contribution of this paper is three-fold:

\begin{itemize}
\item We identify conditions for the underlying MDP ensuring that 
the AVaR MDP problem is well defined when the total cost criterion is used.
\item We define a surrogate MDP problem that can be efficiently solved and whose solution approximates the optimal policy for the original problem with arbitrary
precision.
\item We validate our findings on a rapid robotic deployment task where the objective is to maximize the mission success rate under a given temporal deadline \cite{CarpinASME2014,carpinIROS2014}.
\end{itemize}

The rest of the paper is organized as follows. We discuss related work in Section \ref{sec:related} and provide some background about risk metrics and MDPs in Section \ref{sec:preliminaries}. In Section \ref{sec:formulation}
we formulate the risk-averse, total cost MDP problem we wish to solve.
In Section \ref{sec:risk_aversion} we propose and analyze an approximation strategy for the problem,
 and in Section \ref{sec:AVAR} we provide an algorithmic solution. Simulation results for a rapid deployment problem are given in Section~\ref{sec:simulations},
and conclusions and future work are discussed in Section~\ref{sec:conclusions}.

\section{Related Work}\label{sec:related}
For a general introduction to MDPs the reader is referred to textbooks such as \cite{BertsekasDPVol12} or more recent collections
such as \cite{HanbookMDP}. As pointed out in the introduction,  risk aversion in MDPs has been studied for over four decades, with earlier efforts focusing on exponential utility \cite{Howard1972Risk}, mean-variance \cite{sobel_variance_1982}, and percentile risk criteria \cite{Filar95PP}. 
With regard to mean-variance optimization in MDPs, it was recently 
shown that computing an optimal policy under a variance constraint is  NP-hard \cite{mannor2013}.
Recently,  average value at risk was introduced in \cite{RockafellarCVAR} in order to model the tail risk of a random outcome and to address some key limitations of the prevailing value-at-risk metric. 
Efficient methods to compute AVaR are discussed  in \cite{RU2002}. 
Leveraging the recent strides in AVaR risk modeling, there have been a number of efforts aimed at embedding the AVaR risk metric into risk-sensitive MDPs. In \cite{Ott2011} the authors address the problem of minimizing the AVaR of the discounted cost over a finite and an infinite horizon, and propose a dynamic programming approach based on state augmentation. Similar techniques can be found in \cite{borkar2014risk},
where the authors propose a dynamic programming algorithm for finite-horizon, AVaR-constrained MDPs. The algorithm is proven to asymptotically converge to an optimal risk-constrained policy. However, the algorithm involves computing  integrals over continuous variables (Algorithm 1 in \cite{borkar2014risk}) and, in general, its implementation appears quite challenging. 
A different approach is taken by \cite{prashanth2014policy,tamar2015optimizing,chow2014cvar}
where the authors consider a finite dimensional parameterization of the control  policies, and show that
  an AVaR MDP can be optimized to a 
\emph{local} optimum using stochastic gradient descent (policy gradient). 
However this approach imposes additional restrictions to the policy space and in
 general policy gradient algorithms only converge to a local optimum.
Haskell and Jain recently considered the 
problem of risk aversion in MDPs using a framework based on occupancy measures \cite{Jain2014} (closely connected to our recent works where constrained MDPs
are used to solve the multirobot rapid deployment problem
 \cite{CarpinASME2014,carpinIROS2014}). 
While their findings are only valid for the case where an infinite horizon discounted cost
criterion is considered, the solution we propose
 uses some of the ideas introduced in \cite{Jain2014}.

\section{Preliminaries}\label{sec:preliminaries}
In this section we summarize some known concepts about risk metrics and MDPs. The reader is
referred to the aforementioned references for more details.
\subsection{Risk}
Consider a probability space $\mathcal{S}=(\Omega,\mathcal{F},\mathbb{P})$,
and let $L^{\infty}$ be the space of all essentially bounded random variables on $\mathcal{S}$.
A {\em risk function} (or risk metric) 
$\Gamma: L^{\infty} \rightarrow \mathbb{R}$ 
 is a function that maps an
uncertain outcome  $Y \in L^{\infty}$ onto the real line $ \mathbb{R}$. 
A risk function that is particularly popular in many financial applications is the {\em value at risk}.
For $\tau \in (0,1)$ the {\em value at risk} of $Y\in L^{\infty}$ at level $\tau$ is defined as
\[
\textrm{VaR}_{\tau}(Y) := \inf\{\eta \in \mathbb{R} : \Pr(Y\leq \eta) \geq \tau \}.
\] 
Here VaR$_\tau(Y)$ represents the percentile value of outcome $Y$ at confidence level $\tau$.
Despite its popularity, VaR$_\tau$ has a number of limitations. 
In particular, VaR is not a {\em coherent} risk measure~\cite{artzner1999coherent} and thus suffers from being unstable (high fluctuations under perturbations) when $Y$ is not normally distributed. More importantly it does not quantify the losses that might be incurred beyond its value in the $\tau$-tail of the distribution~\cite{RockafellarCVAR}.
An alternative measure that overcomes most shortcomings of VaR is the {\em average value at risk}, defined as
\[
\textrm{AVaR}_{\tau}(Y) := \frac{1}{1-\tau}\int_{\tau}^1 \textrm{VaR}_{t}(Y)dt,
\]
where $\tau \in (0,1)$ is the confidence level as before. Intuitively,  $\textrm{AVaR}_{\tau}$ is the expectation of $Y$ in the conditional distribution of its upper $\tau$-tail. For this reason, it can be interpreted as a metric of ``how bad is bad."  $\textrm{AVaR}_{\tau}$ can be equivalently written as \cite{RU2002}
\begin{equation}\label{eq:equivalent}
\textrm{AVaR}_{\tau}(Y) = \min_{s\in \mathbb{R}}\left\{s+\frac{1}{1-\tau} \mathbb{E}[(Y-s)^+]\right\},
\end{equation}
where $x^+:=\max(x,0)$.
This paper relies extensively on Eq. \eqref{eq:equivalent} and aims at devising efficient methods to approximate  the 
expectation in Eq. \eqref{eq:equivalent}  when the random variable $Y$ is the total cost of an MDP.
 
Furthermore, it has been recently shown in \cite{YC-AT-SM-MP:15} that optimizing the CVaR of total reward is equivalent to optimizing the worst-case (robust) expected total reward of a system whose model uncertainty is subjected to a trajectory budget. This finding corroborates the fact that a CVaR risk metric models both the variability of random costs, as well as the robustness to system transition errors.
  

\subsection{Total Cost, Transient Markov Decision Processes}
For a finite set $S$, let  $\mathbb{P}(S)$ indicate the set of mass distributions with support on 
$S$. A finite, discrete-time Markov Decision Process (MDP) is a tuple $\mathcal{M}=(X,U,\Pr,c)$ where
\begin{itemize}
\item $X$, the state space,  is a finite set comprising  $n$ elements.
\item $U$, the control space, is a collection of $n$ finite sets $\{U(x_i)\}_{i=1}^n$. Set $U(x_i)$, $i=1,\ldots, n$, represents the  actions that can be applied when in state  $x_i\in X$. The set of
allowable state/action pairs is defined as 
\[
\mathcal{K} := \{ (x,u) \in X\times U~|~u\in U(x) \}.
\]
\item $\Pr(y|x,u):\mathcal{K} \rightarrow \mathbb{R}$ is the transition probability from state $x$ to state $y$  when action $u \in U(x)$ is applied. According to our
 definitions, $\Pr(\cdot|x,u) \in \mathbb{P}(X)$.
\item $c: \mathcal{K}\rightarrow \mathbb{R}_{\geq 0}$ is a non-negative cost function. Specifically, $c(x,u)$ is the cost
incurred when executing action $u\in U(x)$ at state $x$.
\end{itemize}
Let  $\overline{K}:= \max_{(x,u)\in \mathcal{K}}c(x,u)$ and note that the maximum is attained as $\mathcal K$ is a finite set. 
Define the set $\mathcal{H}_t$ of admissible histories up to time
$t$ by $\mathcal{H}_t :=  \mathcal{K} \times  \mathcal{H}_{t-1}$, for $t\geq 1$, and $\mathcal{H}_0 := X$.
An element of $\mathcal{H}_t $  has the form $x_0,u_0,x_1,\dots,x_{t-1},u_{t-1},x_t$, and records all states traversed
and actions taken up to time $t$. In the most general case
 a policy is a function $\pi: \mathcal{H}_t \rightarrow \mathbb{P}(U(x_t))$,
i.e., it decides which action to take in state $x_t$ considering the
entire state-action history. Note that according to this definition a policy is
in general randomized. Let $\Pi$ be the set of all policies,
i.e., including history-dependent, randomized policies.
It is well
known that in the standard MDP setting where an expected cost is minimized
 there is no loss of optimality in restricting  
the optimization over deterministic, stationary Markovian policies, i.e., policies of the type $\pi:X\rightarrow U$.
However, in the risk-averse setting one needs to consider the more general class of history-dependent policies \cite{altman1999constrained}. This is achieved through a state augmentation process
described later.

Following  \cite{Jain2014}, we define the countable space $(\Omega,\mathbb{B}) := (\mathcal{K}^{\infty},\mathbb{B}(\mathcal{K}^{\infty}))$, where $\mathcal{K}^{\infty} = \mathcal{K}\times \mathcal{K} \times \mathcal{K} \times\cdots$, is the sample space and $\mathbb{B}(\mathcal{K}^{\infty})$ is the Borel field on $\mathcal{K}^{\infty}${}. Specific trajectories in the MDP  are written as $\omega \in \Omega$, and we denote by $x_t(\omega)$ and $u_t(\omega)$ the state
and actions at time $t$ along trajectory $\omega$. In general the exact initial state $x_0$ is unknown. Rather it is described by an 
initial mass distribution $\beta$ over $X$, i.e., $\beta \in \mathbb{P}(X)$.
A policy $\pi$ and initial distribution $\beta$ induce a probability distribution over $(\Omega,\mathbb{B})$,
that we will indicate as $\Pr_{\beta}^{\pi}$.

In this paper we focus on transient total cost MDPs, defined as follows. Consider
a partition of $X$ into sets $X^T$ and $M$, i.e., $X = X^T \cup M$ and  $X^T \cap M = \emptyset$. A \emph{transient} MDP is an MDP 
where  \emph{each} policy $\pi$ satisfies the 
following two properties:
\begin{itemize}
\item $\sum_{t=0}^{\infty}\Pr_{\beta}^{\pi}[x_t=x] < \infty$ for each $x\in X^T$, i.e., the state will eventually enter  set $M$, and 
\item $P(y|x,u) = 0$ for each $x\in M$, $y\in X^T$, $u\in U(x)$, i.e., once the state enters $M$ it cannot leave it.
\end{itemize}

A transient, \emph{total cost} MDP is a transient MDP where 
\begin{itemize}
\item $c(x,u) = 0$ for each $x \in M$, i.e., once the state enters $M$ no additional cost is incurred,
\item the cost associated with each trajectory $\omega$ is given by
\[
c(\omega) := \sum_{t=0}^{\infty} c(x_t(\omega),u_t(\omega)).  
\]
\end{itemize}
Note that the cost $c(\omega)$ is a random variable  depending on both the policy $\pi$ and the initial distribution $\beta$. The name total cost stems from the fact that an (undiscounted) cost is incurred throughout the ``lifetime" of the system (i.e., until the state hits the absorbing set).

Transient, total cost MDPs (closely related to stochastic shortest path problems, e.g., \cite{DPB-JNT:91,HY-DPB:12}) represent an alternative to the more commonly used
discounted, infinite-horizon MDPs or finite horizon MDPs. As outlined in the introduction, for 
many robotic applications the total cost, i.e., $c(\omega)$,  is the most appropriate 
cost function.  We justify this statement by noting that most robotic tasks have 
finite duration but such duration is usually not  known in advance. 
In these circumstances the finite horizon cost is inappropriate because
one cannot define the length of the finite horizon up front. Similarly,
the discounted infinite horizon cost is also ill suited because the task
does not continue forever and the cost will not be exponentially 
discounted over time.

Without loss of generality, we assume that set $M$ 
consists of a single absorbing state $x_M$ equipped with a single action $u_{x_M}$, i.e., $M=\{x_M\}$ and
$U(x_M) = u_{x_M}$ with $\Pr(x_M|x_M,u_{x_M})=1$. In the following, with a slight abuse of notation, we denote by  $\mathcal{K}$ the set $\{ (x,u) \in X^T\times U~|~u\in U(x) \}$,
i.e., we exclude the absorbing state from the definition of $\mathcal{K}$. 
Moreover, we assume that for a transient, total cost MDP 
$\beta(x_M)=0$, i.e., the probability of starting at the absorbing state 
is zero. In fact, whenever $x_0 = x_M$ the resulting state
trajectory will deterministically remain in $x_M$, 
and the corresponding cost is zero.

\section{Problem formulation}
\label{sec:formulation}
Our problem formulation relies on the following two technical assumptions
necessary to establish an a-priori upper bound
on the total cost incurred by any trajectory obtained under any policy, and to 
define an approximate problem that can be efficiently solved.
 
The first assumption simply requires that all costs in the transient
states are positive  (recall that we excluded $x_M$ when re-defining $\mathcal{K}$.)
As it will be shown later, this assumption ensures a non-zero discretization step 
when approximating the cumulative cost  accrued by a system throughout the 
trajectory $\omega$ until it is absorbed in $x_M$.
\begin{assumption}[Positivity of costs]\label{ass:nn}  
All costs in $\mathcal{M}$ except for state $x_M$ are positive and bounded,
i.e., $\underline{K} := \min_{(x,u)\in \mathcal{K}}c(x,u)>0$.
\end{assumption}

When considering cost criteria like finite horizon or discounted infinite 
horizon with a finite state space, an a-priori upper bound on the accrued cost can be immediately
established assuming that all costs are finite (a fact crucially exploited in \cite{Jain2014}). However, the situation is more complex
when considering the total cost case, because without introducing 
further hypotheses on the structure of the MDP a malicious adversary could
establish an history-dependent policy capable of invalidating any
a-priori established bound on the cost\footnote{First note that we are seeking
a uniform upper bound for all possible policies, including history 
dependent policies. Hence, given a tentative bound $B$, 
in the general case one could devise a 
history dependent policy ensuring that every trajectory generated by the policy is not absorbed in $x_M$
in less than $B/\underline{K}$ steps, thus invalidating the bound. }.
The second assumption then adds a ``global reachability structure" to the MDP problem.
To this end, in the following, it will be useful to consider the Markov 
Chain generated by the MDP when an input is selected for each state. 
For an MDP $\mathcal{M}$, select $u_1\in U(x_1), \dots, u_n\in U(x_n)$.
The selected inputs and the transition probabilities in $\mathcal{M}$ define
a finite Markov Chain that we indicate as $\mathcal{MC}_{u_1,  \ldots, u_n}$.
The state space of $\mathcal{MC}_{u_1, \ldots, u_n}$ is equal to $X$
and for two states $x_i,x_j \in X$ the transition probability $\Pr_{i,j}$ is
defined as $\Pr_{i,j} = \Pr(x_j|x_i,u_i)$ where $u_i\in U(x_i)$ is the input
selected in the definition of $\mathcal{MC}_{u_1,  \ldots, u_n}$ and $\Pr$ is the transition probability
of the associated MDP.
	
\begin{assumption}[Reachability of MDP]\label{ass:irr}   
Let $\mathcal{MC}_{u_1,  \ldots, u_n}$ be the Markov chain induced by the
$n$ inputs $u_i\in U(x_i)$.
Then the absorbing state $x_M$, under Markov chain  $\mathcal{MC}_{u_1,  \ldots, u_n}$, is reachable from any state $x\in X^T$, for all $u_1\in U(x_1),  \ldots, u_n\in U(x_n)$.
\end{assumption}
We recall that a state $j$ in a Markov chain is said reachable from another state $i$ if there exists an integer $k\geq 1$ such that the probability that the chain
will be in state $j$ after $k$ transitions  is positive \cite{gallager2013stochastic}.
Note that when Assumption \ref{ass:irr} holds, under every policy there is a 
path of non-zero probability connecting every state to $x_M$.
 Therefore, it is impossible to devise a policy that prevents sure absorption for 
 an arbitrary number of steps. This holds for all policies, including  history dependent policies (see Figure \ref{fig:Assumption2}).

\begin{figure}
\centering
\includegraphics[width=0.7 \columnwidth]{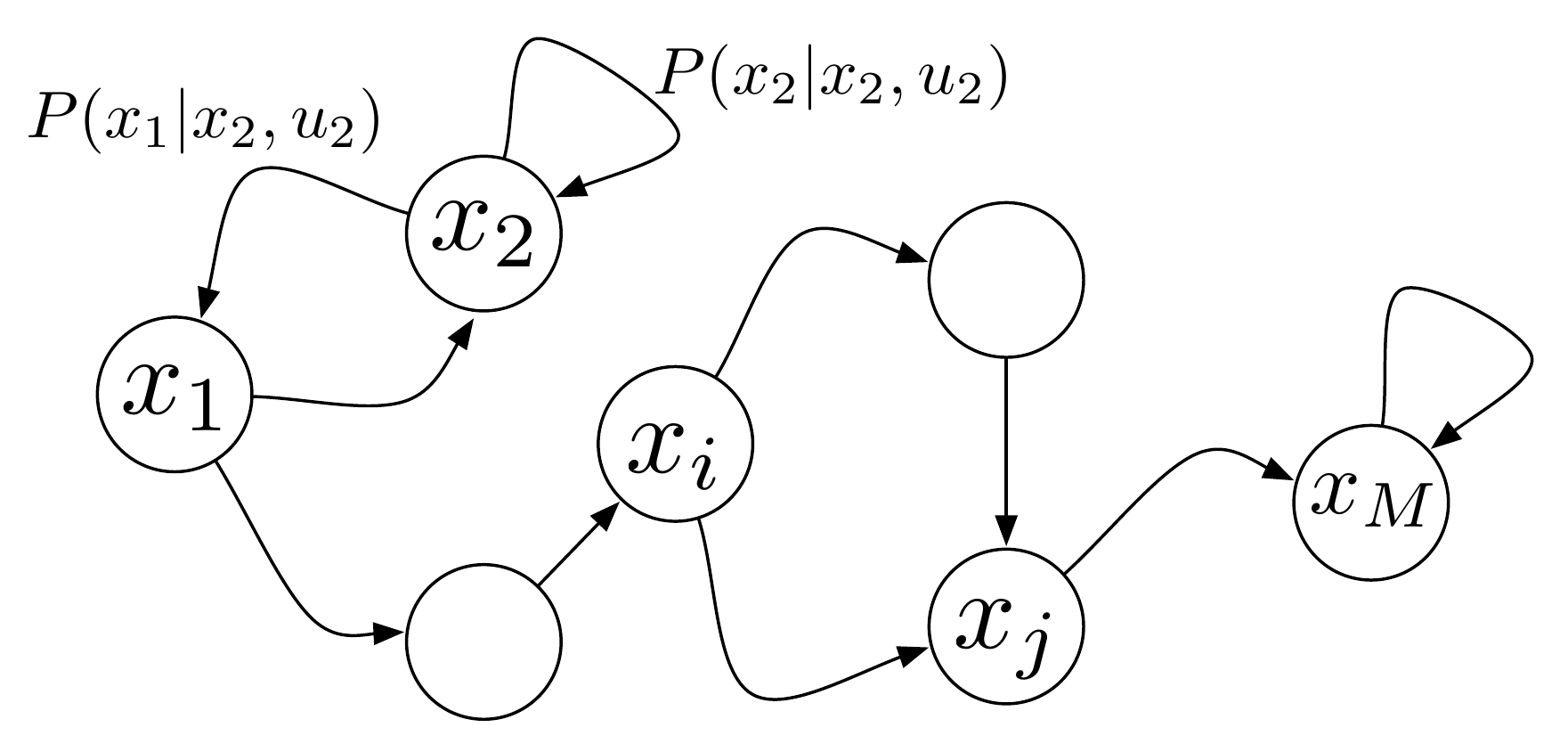}
\caption{Meaning of Assumption \ref{ass:irr}: after one input has been chosen for every state, an associated Markov Chain  
$\mathcal{MC}_{u_1,  \ldots, u_n}$ is defined.  In this Markov Chain, the absorbing state $x_M$ is reachable from every state,
i.e., at every state there is a path of non-zero probability to $x_M$. The probability of a path is given by the product of the probabilities
of its edges, i.e., the probability of the path $x_i \to x_j \to x_M$ is $\Pr(x_j|x_i,u_i)\Pr(x_M|x_j,u_j)$. This requirement is imposed 
for every possible choice of the inputs and every policy.}
\label{fig:Assumption2}
\end{figure}

Building upon the previous material, 
we can now define the problem we aim to solve in this paper:
\begin{quote}{\bf Risk-Averse, total cost MDP} -- Given a transient total cost MDP satisfying Assumptions \ref{ass:nn}-\ref{ass:irr}, and an initial distribution $\beta$, determine a policy 
$\pi$ that minimizes $\textrm{AVaR}_{\tau}(c(\omega))$, i.e., find
\begin{equation}\label{eq:optim}
\pi^{*} \in \text{arg} \min_{\pi \in \Pi} \textrm{AVaR}_{\tau}(c(\omega)).
\end{equation}
\end{quote}
Note that, under the assumption of transient total cost MDP,  one can easily verify that  $\mathbb{E}[c(\omega)]<\infty$. Since, by equation \eqref{eq:equivalent}, $\textrm{AVaR}_{\tau}(c(\omega)) \leq 1/(1-\tau) \mathbb{E}[c(\omega)]$, one obtains $\textrm{AVaR}_{\tau}(c(\omega))<\infty$  for all $\omega$, as well.
However, to derive an optimization algorithm for the computation of $\pi^{*} $
it is necessary to formulate an a-priori 
upper bound for the  optimal cost in \eqref{eq:optim}.
Assumptions \ref{ass:nn} and \ref{ass:irr} are introduced to ensure 
that such bound exists and can  be computed.

\section{Approximation Strategy}\label{sec:risk_aversion}
In this section we study an approximation strategy for the risk averse total cost MDP in equation \eqref{eq:optim}. 
Similar to the method presented in \cite{Jain2014}, we aim at solving the problem by using the concept of \emph{occupation measures}.
 However, unlike for the cases studied in \cite{Jain2014}, in total cost MDPs an explicit upper bound for the accrued cost is not 
available, which makes the solution strategy in \cite{Jain2014} not applicable. Our strategy is to find a surrogate to 
problem \eqref{eq:optim}. By imposing an effective horizon, we construct a total cost MDP with time-out and recast this 
 problem into a bilinear programming problem. Furthermore we characterize the  sub-optimality gap for such  surrogate approximation. We start with a technical result characterizing the convergence rate to the absorbing state.

\subsection{Convergence rate to the absorbing state}
Consider a selection of  inputs $u_1\in U(x_1), \ldots, u_n\in U(x_n)$ and the corresponding 
Markov chain $\mathcal{MC}_{u_1,  \ldots, u_n}$. For each state $x\in X^T$, let 
$\text{MinimumPath}_{x\to x_M}(\mathcal{MC}_{u_1,  \ldots, u_n})$ denote the simple (i.e., without cycles) path
 from $x$ to $x_M$ of lowest, strictly positive probability.
Note that  $\text{MinimumPath}_{x\to x_M}(\mathcal{MC}_{u_1,  \ldots, u_n})$
exists due to Assumption \ref{ass:irr}.
Let $\Pr(\text{MinimumPath}_{x\to x_M}(\mathcal{MC}_{u_1,  \ldots, u_n}))$ be 
the probability of the path, i.e., the product of the probabilities of
all the transitions along the path.
Since there are $n$ nodes and by definition the path is simple,
 $\text{MinimumPath}_{x\to x_M}(\mathcal{MC}_{u_1,  \ldots, u_n})$
includes at most $n-1$ transitions between $n$ nodes. 
Let
\[
\gamma := \min_{u_k \in U(x_k), \atop k=1\ldots,n}  \min_{x\in x^T}  \Pr(\text{MinimumPath}_{x\to x_M}(\mathcal{MC}_{u_1, \ldots, u_n})).
\]

%
%

Note that the minimum is achieved as the minimization is over a finite set,
 and that $\gamma$ is strictly positive due to Assumption \ref{ass:irr}. 
 The constant $\gamma$  lower bounds the probability that, under \emph{any} 
 policy $\pi\in\Pi$, the absorbing state is reached in no more than $n$ steps, 
 from \emph{any} state $x\in X^T$.
We are now in a position to characterize the convergence rate to the absorbing state.
\begin{lemma}[Number of stages to reach the absorbing set]\label{lem:tech1}

For any policy $\pi\in\Pi$ and initial distribution $\beta$,
\[
\Pr_{\beta}^{\pi} [x_{kn}\neq x_M]\leq(1 - \gamma)^k,\,\, \forall k\in \mathbb{N}.
\]

\end{lemma}
\begin{proof}
The claim is proven by induction on $k$. Base case: we prove that  
\[
\Pr_{\beta}^{\pi}[x_n\neq x_M] \leq 1- \gamma.
\]
Indeed, $\Pr_{\beta}^{\pi}[x_n\neq x_M] = \sum_{x\in X_T} \Pr_{\beta}^{\pi}[x_n\neq x_M|x_0 = x]\Pr_{\beta}^{\pi}[x_0 =x]$. Because of Assumption \ref{ass:nn}, for any policy $\pi$, $\Pr_{\beta}^{\pi}[x_n  =  x_M] \geq \gamma$, and the base case follows.

For the inductive step, assume that $\Pr_{\beta}^{\pi}[x_{kn}\neq x_M]<(1 - \gamma)^k$, for some $k>1$. Then, $\Pr_{\beta}^{\pi}[x_{(k+1)n}\neq x_M] = \Pr_{\beta}^{\pi}[x_{(k+1)n} \neq x_M |x_{kn}\neq x_M]\Pr_{\beta}^{\pi}[x_{kn}\neq x_M]$. By definition of $\gamma$, 
\[
\Pr_{\beta}^{\pi}[x_{(k+1)n} \neq x_M |x_{kn}\neq x_M] \leq (1-\gamma)^{k+1},
\]
and the claim follows.
\end{proof}

\subsection{Surrogate problem and approximation bounds}

Our solution strategy is to solve a {\em surrogate} problem, whereby after a
 deterministic number $d\in \mathbb{N}$ of steps, the state moves to the absorbing state $x_M$ surely. In other words, $d$ acts as a ``timeout" for the MDP 
 problem. 
 The surrogate problem is simpler to solve, and we will show in the following that
 its solution can approximate the solution of the original problem with
 arbitrary precision\footnote{An alternative strategy would be to investigate reductions of problem \eqref{eq:optim} to an equivalent risk-averse, discounted, infinite-horizon problem by using, e.g., the results recently presented in \cite{EAF-JH:15}, and then apply the approach in \cite{Jain2014} to the reformulated problem. This is an interesting  direction left for future research.}.
 Denote by $c^{[d]}(\omega)$ the total cost for such surrogate problem. 
 Additionally, for the original problem, let $t^*(\omega)$ denote the {\em absorbing 
 time}, i.e., the time at which   the state reaches $x_M$. If $t^*(\omega)\leq d$, then the two processes coincide and then
  $c^{[d]}(\omega)=c(\omega)$. 
 Otherwise, for each trajectory $\omega$ such that 
 $t^*(\omega)> d$, the random process is stopped after $d$ steps, and the state goes, 
 deterministically, to  $x_M$ at stage $d+1$. In such a case one has 
 $c^{[d]}(\omega)\leq c(\omega)$. 

We want to characterize the relation between $\textrm{AVaR}_{\tau}(c^{[d]}(\omega))$ (i.e., the risk for the surrogate problem) and $\textrm{AVaR}_{\tau}(c(\omega))$ (i.e., the risk for the original problem). To this end, let $c_d(\omega)$ be the total cost for the original problem up to time $d$, i.e., 
\[
c_d(\omega) := \sum_{t=0}^d\, c(x_t(\omega), u_t(\omega)).
\]
The following lemma shows the equivalence between $c^{[d]}(\omega)$ and $c_d(\omega)$.

%

 \begin{lemma}[Correspondence of costs]\label{lem:tech2}
For any policy $\pi \in \Pi$ and any trajectory $\omega$, $c^{[d]}(\omega) = c_{d}(\omega)$.
\end{lemma}
\begin{proof}
Given a policy $\pi$, for any trajectory $\omega$, the cost cumulated up to time $d$ is the same for both the original and the surrogate problem. After time $d$, both $c_{d}(\omega)$ and $c^{[d]}(\omega)$ do not cumulate any additional cost, then the claim follows.
\end{proof}

The following theorem represents the main result of this section
\begin{theorem}[Suboptimality bound]\label{them:sub}
The surrogate problem approximates the original problem according to 
\begin{equation*}
\begin{split}
& \min_{\pi} \textrm{AVaR}_{\tau}(c^{[d]}(\omega)) \leq  \min_{\pi} \textrm{AVaR}_{\tau}(c(\omega))\\
 &\qquad  \qquad \leq  \min_{\pi} \textrm{AVaR}_{\tau}(c^{[d]}(\omega)) +   \frac{n\overline{K}}{1 - \tau}\,  \frac{(1- \gamma)^{\lfloor (d+1)/n\rfloor}}{\gamma}.
 \end{split}
\end{equation*}

\end{theorem}
\begin{proof}
The left inequality is proven by noticing that $\min_{\pi} \textrm{AVaR}_{\tau}(c(\omega)) \geq \min_{\pi} \textrm{AVaR}_{\tau}(c_d(\omega)) = \min_{\pi} \textrm{AVaR}_{\tau}(c^{[d]}(\omega))$, where  the equality follows from Lemma \ref{lem:tech2}. 

We now prove the right inequality. For any $s\in \mathbb{R}$ and policy $\pi$, one has 
\begin{align}
 \mathbb{E}[&(c(\omega)-s)^+]=\mathbb{E}[(c(\omega)-s)^+\mid t^*(\omega)\leq d  ]\mathbb P(t^*(\omega)\leq d )\nonumber\\
 &+\mathbb{E}[(c(\omega)-s)^+\mid t^*(\omega)> d ]\mathbb P(t^*(\omega)> d ).\label{eq:decompose}
 \end{align}

Let $c_l(\omega) := \sum_{t=d+1}^{\infty} c(x_t(\omega),u_t(\omega))$ be the tail cumulated cost, and, as before, $c_d(\omega) := \sum_{t=0}^{d} c(x_t(\omega),u_t(\omega))$. Since the function $x\to x^+$ is sub-additive, i.e., $(x+y)^+\leq x^++y^+$ and  the expectation operator preserves monotonicity, one obtains the inequality 
\[
\begin{split}
\mathbb{E}[(c&(\omega)-s)^+\mid t^*(\omega)> d ] \\
&= \mathbb{E}[(c_d(\omega)+ c_l(\omega)-s)^+\mid t^*(\omega)> d ]\\
&\leq\mathbb{E}[(c_d(\omega)-s)^+ \mid t^*(\omega)> d ] +\mathbb{E}[c_l(\omega)  \mid t^*(\omega)> d ].
\end{split}
\] 
Furthermore,  for each  trajectory in the event set $\{\omega:t^*(\omega)\leq d \}$, one has
\[
\mathbb{E}[(c(\omega)-s)^+\mid t^*(\omega)\leq d  ]= \mathbb{E}\left[\left(c_d(\omega)-s\right)^+\mid t^*(\omega)\leq d  \right].
\] 
Collecting the results so far, one has the following inequalities:
\begin{equation}\label{eq:intermediate}
\begin{split}
&\mathbb{E}[(c(\omega)-s)^+]\\
& \leq \mathbb{E}[(c_d(\omega)-s)^+\mid t^*(\omega)\leq d  ]\mathbb P(t^*(\omega)\leq d )  \\
&\qquad + \mathbb{E}[(c_d(\omega)-s)^+ \mid t^*(\omega)> d ] \mathbb P(t^*(\omega)> d) + \\
&\qquad +  \mathbb{E}\left[c_l(\omega)\mid t^*(\omega)> d \right]\mathbb P(t^*(\omega)> d)\\
& = \mathbb{E}[(c_d(\omega)-s)^+ ]+\mathbb{E}\left[c_l(\omega)\mid t^*(\omega)> d \right]\mathbb P(t^*(\omega)> d)\\
& \leq \mathbb{E}[(c_d(\omega)-s)^+ ]+\mathbb{E}\left[c_l(\omega)\right].
\end{split}
\end{equation}

Equation \eqref{eq:intermediate} implies
\begin{equation*}\label{eq:intermediate_2}
\begin{split}
&\quad\min_{\pi} \textrm{ AVaR}_{\tau}(c(\omega))\\
& = \min_{\pi} \, \min_{s\in \mathbb{R}}\left\{s+\frac{1}{1-\tau} \mathbb{E}[(c(\omega)-s)^+]\right\}\\
& \leq \min_{\pi} \, \min_{s\in \mathbb{R}}\left\{s+\frac{1}{1-\tau} \left(\mathbb{E}[(c_d(\omega)-s)^+ ]+\mathbb{E}\left[c_l(\omega)\right]\right)\right\}\\
& \leq \min_{\pi} \, \min_{s\in \mathbb{R}}\left\{s+\frac{1}{1-\tau} \left(\mathbb{E}[(c_d(\omega)-s)^+ ]\right)\right\} \\
&\qquad +\max_{\pi} \, \frac{1}{1-\tau}\mathbb{E}\left[c_l(\omega)\right]\\
& = \min_{\pi} \, \min_{s\in \mathbb{R}}\left\{s+\frac{1}{1-\tau} \left(\mathbb{E}[(c^{[d]}(\omega)-s)^+ ]\right)\right\} \\
&\qquad +\max_{\pi} \, \frac{1}{1-\tau}\mathbb{E}\left[c_l(\omega)\right]\\
& =  \min_{\pi}  \textrm{ AVaR}_{\tau}(c^{[d]}(\omega)) + \max_{\pi} \, \frac{1}{1-\tau}\mathbb{E}\left[c_l(\omega)\right],
\end{split}
\end{equation*}
where the second to last equality follows from Lemma \ref{lem:tech2}. We are left with the task of upper bounding $\mathbb{E}\left[c_l(\omega)\right]$. To this purpose, one can write
\[
\mathbb{E}\left[c_l(\omega)\right]\leq \overline{K} \sum_{t=d+1}^{\infty} \, \Pr_{\beta}^{\pi}(x_t \neq x_M).
\]

Note  that  the result in Lemma \ref{lem:tech1} implies
\[
\begin{split}
\sum_{t=d+1}^{\infty}& \Pr_{\beta}^{\pi}[x_{t}\neq x_M] \leq\!\! \!\sum_{k=\lfloor (d+1)/n\rfloor}^{\infty}\!\!\!\!\sum_{t=kn}^{(k+1)n-1}\Pr_{\beta}^{\pi}[x_{t}\neq x_M]\\
&\leq \sum_{k=\lfloor (d+1)/n\rfloor}^{\infty} n(1-\gamma)^k
  = n \frac{(1- \gamma)^{\lfloor (d+1)/n\rfloor}}{\gamma}.
\end{split}
\]
The claim then follows immediately, as the above upper bound is policy-independent.
\end{proof}

Note that according to Theorem \ref{them:sub}, as $d\to \infty$, the optimal cost of the surrogate problem recovers the optimal cost of the original problem, i.e., the surrogate problem provides a consistent approximation to the original problem, with a sub optimality factor that is computable from problem data.

Lemma \ref{lem:tech2} and Theorem \ref{them:sub} ensure that $c^{[d]}(\omega)$ can  approximate $c(\omega)$ with arbitrary 
precision for a sufficiently large value of $d$. 
In the next section we then show how to solve the minimization problem: 
\begin{equation}\label{eq:toapprox}
\min_{\pi} \textrm{AVaR}_{\tau}(c^{[d]}(\omega)).
\end{equation}

\section{Solution Algorithm}\label{sec:AVAR}

Leveraging the surrogate problem from the previous section, we can now
adapt  the results proposed in \cite{Jain2014} to solve problem \eqref{eq:optim}.
An essential step to solve this optimization problem is to compute $\mathbb{E}[(c^{[d]}(\omega)-s)^+]$, which entails deriving 
the probability distribution for the possible costs generated by the random 
variable $c^{[d]}(\omega)$. This problem can be solved by suitably augmenting 
the state space as described in the following, and then using occupancy measures.
In the space of  occupancy measures, an optimal policy is determined
through the solution of a bilinear program, as explained below.
For a given policy $\pi$ and initial
distribution $\beta$, we define the occupancy measure for $(x,u)\in \mathcal{K}$ as
\[
\rho(x,u) = \sum_{t=0}^{\infty}\Pr{_{\beta}^{\pi}}[x_t(\omega)=x,u_t(\omega)=u)].
\]
Note that $\rho(x,u)$ is non negative but is in general not a probability itself.
In the following we will use occupancy measures to determine the probability distribution of the total costs $c^{[d]}(\omega)$ and then
to compute the needed expectation.
According to the definition,
occupancy measures depend on the policy $\pi$ and the initial distribution $\beta$. 
Given an absorbing MDP $\mathcal{M}=(X,U,\Pr,c)$, we define a new {\em state-augmented} absorbing MDP with  additional state components that 
track the cumulated total cost and current stage.
 Although the original MDP $\mathcal{M}$ is finite and absorbing, the set of 
costs $c^{[d]}(\omega)$ generated by all possible policies can be very large,
 and this can subsequently
lead to a linear program with an unmanageable number of decision variables.
To counter this problem, we introduce a  discretized approximation 
for $c^{[d]}(\pi,\beta)$ whose error can be arbitrarily bounded.
To this end, we set $
\zeta = \min \left\{ \underline{K},\frac{d\bar{K}}{N'}\right\}$,
where $N'\in \mathbb{N}$ is a parameter describing the desired number of discretized values for the cumulated cost. Due to Assumption \ref{ass:nn}, $\zeta$ is strictly positive.
The effective number of different values is
$
N = \left \lceil \frac{d\bar{K}}{\zeta} \right \rceil.
$
This value may be higher than $N'$ due to our definition of $\zeta$. 
We then define a new MDP
$\mathcal{M}'_N=(X',U',\Pr',c')$ as follows.
Its state space is $X'=X\times \mathbb{N}_{N} \times \mathbb{N}_{d}$, where $\mathbb{N}_N = \{0,1,\dots,N\}$ and  $\mathbb{N}_{d} = \{0,1,\dots,d\}$. 
Elements in the augmented states will be indicated as $(x,y,z)$.
As clarified in the following, the two additional components store the cumulated running cost ($y$) and current stage ($z$). Recall that in the surrogate problem, after $d$ steps, the state is guaranteed to have entered the absorbing set, i.e., it is guaranteed that  $x_d(\omega)=x_M$. 
Thus the value of the $z$ component 
is in $\mathbb{N}_{d} = \{0,1,\dots,d\}$. On the other hand the input
sets are defined as $U'(x,y,z)=U(x)$. $X'$ and $U'$ induce a new set $\mathcal{K}' = \{ (x,y,z,u) ~|~ (x,u)\in \mathcal{K} \wedge y\in \mathbb{N}_{N} \wedge z\in \mathbb{N}_{d} \}$.
The new cost function $c':\mathcal{K}' \rightarrow \mathbb{R}_{\geq 0}$ is $c'(x,y,z,u) = c(x,u)$. The transition probability function is modified as follows:
\begin{align}
& \Pr'((x',y',z')|(x,y,z),u)= \nonumber \\
&
\begin{cases}
\Pr(x'|x,u) &  \textrm{if} ~y' = y  + \left \lfloor \frac{c(x,u)}{\zeta}\right \rfloor \wedge z' =z+1\nonumber  \wedge z'<d\\
1 &  \textrm{if } (x',y',z') =(x_M,y,d) \wedge z=d\\
0 & \nonumber\textrm{otherwise} 
\end{cases}
\end{align}
As evident from the  definition of the new transition function, the new variables included in the state stores the discretized\footnote{To be precise, the discretized running cost is scaled by $\zeta$.} running cost and the stage. 
Consistently with our definition of the surrogate problem,
the revised transition function 
includes a timeout that imposes a transition to the absorbing state $x_M$ after $d$ steps, and
from that point onwards the accrued cost does not change. 
Note also that the additional state components $y$ and $z$
are deterministic functions of the previous state and control input $u$.
Extending the formerly introduced notation, for a given trajectory $\omega$ of $\mathcal{M}'_N$, we write $y_t(\omega)$ for the second component
of the state at time $t$ and $z_t(\omega)$ for the third component. 
Finally, for a given 
initial distribution $\beta$ on $X$, we define the following new initial distribution $\beta'$ on $X'$,
\[
\beta'(x,y,z)=
\begin{cases}
\beta(x) & \textrm{if} ~y=0 ~\wedge z=0\\
0 & \textrm{otherwise} 
\end{cases}.
\]

Note that the properties of $\mathcal{M}$ carry over to $\mathcal{M}'_N$. In particular, if Assumptions \ref{ass:nn}-\ref{ass:irr} 
hold for $\mathcal{M}$ then they hold for $\mathcal{M}'_N$ too,
and, if $\mathcal{M}$  is absorbing, then $\mathcal{M}'_N$ is also absorbing. \zsmargin{Thus we indicate with $X'^T$ its transient set of states.}{Why is this needed? It doesn't appear to be used anywhere. If you keep it, ``its'' is ambiguous. Do you mean $\mathcal{M}'_N$'s set of transient states?}
For a given realization $\omega$, consider now $c_t^{[d]}(\omega) = \sum_{i=0}^tc(x_i,u_i)$, i.e., the true cumulative cost of the
surrogate MDP problem without discretization. 
The following theorem establishes that even though the approximation error introduced by discretizing 
the running cost grows linearly with $t$, it is possible to bound it with arbitrary precision.
\begin{theorem}\label{th:approx}
For each $\varepsilon >0$ and each $t\in\{0,\ldots,d\}$, there exists a discretization step $\zeta$ such that 
$|\zeta y_t(\omega)-c_t^{[d]}(\omega)| < \varepsilon$.
\end{theorem}
{\em Proof}. 
Pick $\zeta = \varepsilon/d$.
Let $e(t) = c_t^{[d]}(\omega)-\zeta y_t(\omega)$ be the approximation error at time $t$. Note
that by definition $e(t)\geq 0$ and $e(0)= c_0^{[d]}(\omega)-\zeta y_0(\omega)=0$.
From the definition of the transition probability function, $P'$, it follows that $e(t+1) \leq e(t)+\zeta$,
which implies $e(d)\leq d\zeta = \varepsilon$. Since for $t>d$ we have
$e(t) = e(d)$, the claim follows. $\Box$\\

A key step towards the solution of problem in \eqref{eq:toapprox} is therefore to derive the statistical 
description of the discretized total cost $y_d(\omega)$ that is used to approximate $c^{[d]}(\omega)$. This objective can be
achieved by exploiting the occupancy measures for the state-augmented MDP $\mathcal{M}'$.
For $(x,y,z,u)\in \mathcal{K}'$, the occupancy measure  on $\mathcal{M}'$ induced by a policy $\pi$ and an initial distribution $\beta$ is given as:
\begin{align}
& \rho(x,y,z,u) = \label{eq:occupancy} \\
&  \sum_{t=0}^{\infty}\Pr{_{\beta}^{\pi}}[x_t(\omega)=x,y_t(\omega)=y,z_t(\omega)=z,u_t(\omega)=u].\nonumber
\end{align}
 The occupancy measure, $\rho$, is a 
vector in  $\mathbb{R}_{\geq 0}^{|\mathcal{K}'|}$, i.e., it is a vector with $|K'|$ non negative components.
 The set of legitimate occupancy vectors is constrained by the initial distribution $\beta$ and defined by the 
policy $\pi$. It is well known \cite{altman1999constrained} that these constraints can be expressed as follows:

\begin{align}
\sum_{(x',y',z')\in {X}'^T} \sum_{u\in A(x',y',z')}\rho(x',y',z',u)[\delta_{(x,y,z)}(x',y',z')- \nonumber \\
P'((x',y',z')|(x,y,z),u)]=\beta(x,y,z)~\forall (x,y,z)\in {X}'^T \nonumber
\end{align}
where $\delta_x(y)=1$ if and only if $y=x$.
For $0 \leq k \leq N$ we introduce random variables $\theta(k)$ with the property
that $\theta(k)=\Pr[y_d(\omega)=k]$.
 This is easily achieved using occupancy measures, i.e.,
$
\theta(k) =  \sum_{(x,y,z,u)\in \mathcal{K}'}I(y = k~\wedge z=d)\rho(x,y,z,u),
$
where $I(\cdot)$ is the indicator function equal to 1 when its argument is true, and 0 otherwise. 
Note that by definition $\theta(k)$ is equal to $\Pr[y_d(\omega)=k]$,
and by Theorem \ref{th:approx} $y_d(\omega)$ approximates $c^{[d]}(\omega)$ with arbitrary precision.
Combining  the above definitions we then get to the following problem whose
solution approximates the solution to \eqref{eq:toapprox}:
\begin{align}\label{eq:CVAR_MDP}
&\min_{\rho,\theta} \min_{s \in [0,\bar{K}d]} s+\frac{1}{1-\tau} \sum_{y \in \mathbb{N}_N}(y-s)^+\theta(y)  \\
&\textrm{s.t.} \nonumber \\
&\sum_{(x',y',z')\in {X}'^T} \sum_{u\in A(x',y',z')}\rho(x',y',z',u)[\delta_{(x,y,z)}(x',y',z')- \nonumber \\
&P'((x',y',z')|(x,y,z),u)]=\beta(x,y,z)~\forall (x,y,z)\in {X}'^T \nonumber\\
 & ~ ~ ~ \theta(k) =\!\!\!\!\!  \sum_{(x,y,z,u)\in \mathcal{K}'}\!\!\!\!\!I(y = k~\wedge z=d) \rho(x,y,z,u),\, 0 \leq k \leq N.\nonumber
\end{align}
When comparing this last optimization problem with \eqref{eq:toapprox}, 
the reader will note that the variable $s$ is constrained in the interval $[0,\bar{K}d]$. Indeed, the objective
function is continuous with respect to $s$, and it is straightforward to verify that the partial derivative of
the objective function with respect to $s$ is negative for $s<0$ and positive for $s>\bar{K}d$. 
The objective function given in Eq. \eqref{eq:CVAR_MDP} is concave with respect to $\theta(y)$ and is defined over 
a convex feasibility set \cite{Jain2014}. 
To the best of our knowledge, there exist no efficient methods to determine the global minimum for this class of problems.
Hence, the problem is approximately solved fixing different values of $s$ within the range $[0,\bar{K}d]$, and then
solving the corresponding linear problem over the optimization variables $\rho$ and $\theta$.
Comparing the problem in Eq. \eqref{eq:CVAR_MDP} with the one  in Eq. \eqref{eq:optim}
 one might initially think that the objective function in  
 Eq. \eqref{eq:optim} does not depend on the policy $\pi$.
However, the dependency on $\pi$ is carried over by the occupancy measure $\rho$, as evident from Eq. \eqref{eq:occupancy}. Moreover,
it is well known from the theory of constrained MDPs \cite{altman1999constrained} that there is a one to one correspondence between
policies and occupancy measures, i.e., every policy defines a unique occupancy measure and every occupancy measure induces a policy.


\section{Numerical Experiments}
\label{sec:simulations}

To illustrate the performance of the proposed algorithm, we adopt the rapid deployment scenario considered in \cite{carpinIROS2014,CarpinASME2014}. 
A graph is used to abstract and model the connectivity of a given map of an environment 
(see, e.g., \cite{carpinIROS2008}).
One robot is positioned at a start vertex and is
tasked to reach the goal vertex within a given temporal deadline while providing some guarantee about its
 probability of successfully completing the task.
When moving from vertex
to vertex, the robot can choose from a set of actions, each trading off 
speed with probability of success. In particular,  actions with
rapid transitions between two vertices have higher probability of failure;
and conversely when the robot moves slowly between two vertices it has a higher 
probability of success. In this scenario, {\em failure} means that the robot
does not move (e.g., fails to pass through an opening), so elapsed time increases
without making progress towards the goal.
With a given temporal deadline $T$ and success probability $P$, the robot is tasked to reach
the target vertex ``safely" (such that the true mission success probability is at least $P$), while satisfying the temporal constraint.
From a design perspective it is of interest to know if there exists a policy $\pi$ 
achieving this objective, and to compute it. If the policy does not exist, it is
of interest to know how to modify the parameters in order to make the
task feasible.

In our previous work we solved this problem by modeling it 
using Constrained Markov Decision Processes (CMDP).
In the CMDP approach, one maximizes the probability of success
while imposing a constraint on the temporal deadline.
 However, this method only returns risk-neutral policies, i.e., the resultant policies only guarantee that the temporal deadline is met in expectation, and there is no explicit control on the tail probability of the constraint.
As a radical departure from the original problem formulation, the AVaR minimization method proposed in Eq. \eqref{eq:optim} searches for a policy that is feasible with respect to the temporal deadline constraint\footnote{For any random variable $Z$ with finite expectation,  $\textrm{AVaR}_\tau(Z)\geq \mathbb E[Z]$ for $\tau\in[0,1]$. Therefore, if the solution to the AVaR minimization problem is bounded above by the temporal deadline, then the corresponding minimizer is also a feasible policy to the original problem.} and systematically controls the worst-case variability of total travel time. Note that a policy with low success probability will have large tail probability in total travel time even if the expected temporal deadline is met. Therefore the optimal policies from AVaR minimization  will have \emph{high success probability}. This motivates the application of AVaR minimization to rapid robotic deployment.
First, note that, in the devised setting, the robot will eventually reach the final goal with positive probability.
However, due to possible failures one cannot put an a-priori bound on the
random total travel time. Therefore, the total cost criterion is indeed a
natural choice for this task. Moreover, Assumptions \ref{ass:nn} and \ref{ass:irr}
are easily justified because the immediate cost function (i.e., time to move) is always positive and
 the global reachability property follows from the graph structure.

To illustrate the performance of risk-averse deployment, two different policies are compared. Here both policies are computed using unconstrained stochastic control methodologies for which the immediate cost is the travel time between two vertices, and the actions correspond to all possible node transitions on the graph.
The first is the classic risk-neutral policy obtained with 
value iteration. The second is a risk averse policy obtained with
the algorithm presented in this paper using $\tau=0.95$.
For each policy, 1000 executions are run, and the distribution of total travel time is reported.
Figures \ref{fig:riskneutral} and \ref{fig:riskaverse} show the
distributions for the two cases. The risk-neutral policy obtains 
a lower expected cost, but has a longer tail, as evidenced
by the 61 instances with a cost larger than or equal to 15 (notice that $T=15$ is the desired time of completion in this example).
Moreover, as evidenced by the shape of the histogram, costs are more spread out.
The risk averse policy, on the other hand, results in less variability as desired. Less than 30 instances have a cost 
larger or equal than 15, a reduction of the weight of the tail by more than one half. 
\begin{figure}[tb]
\centering
\includegraphics[width=0.75\linewidth]{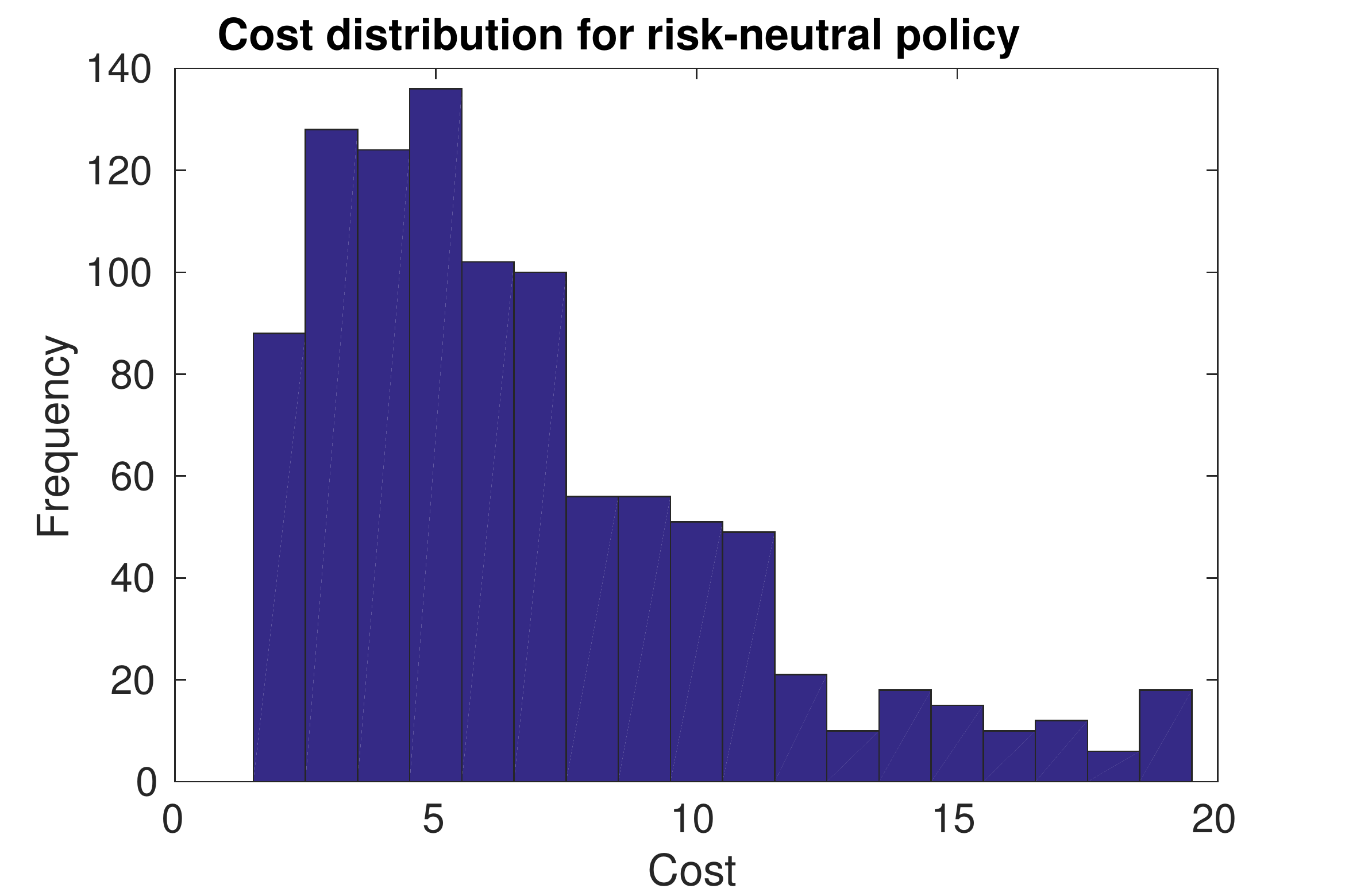}
\caption{Cost distribution for a risk-neutral policy}
\label{fig:riskneutral}
\end{figure}
\begin{figure}[tb]
\centering
\includegraphics[width=0.75\linewidth]{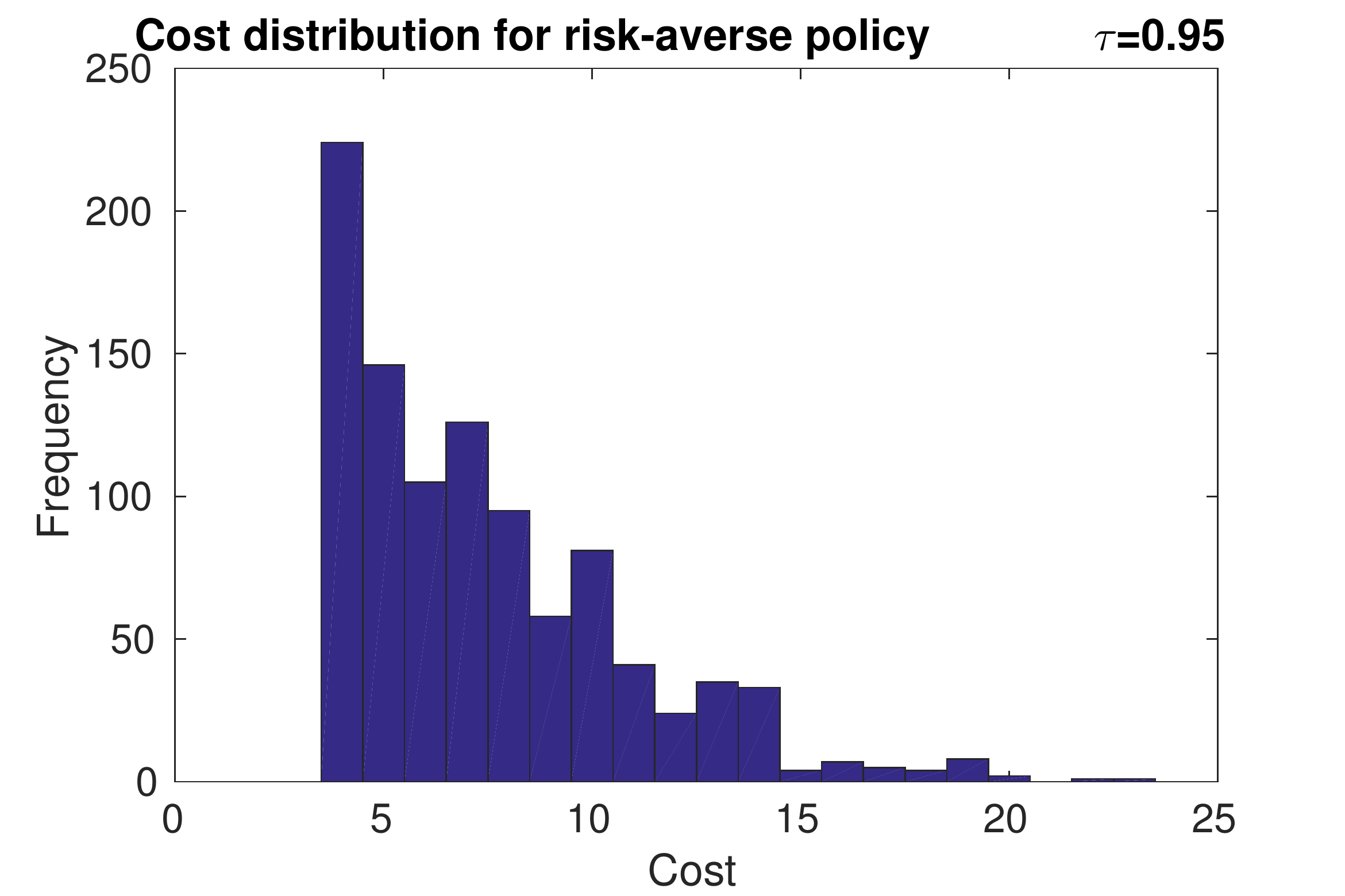}
\caption{Cost distribution for a risk averse policy with $\tau=0.95$.}
\label{fig:riskaverse}
\end{figure}
Importantly, when computing a risk-neutral policy using classic methods like 
policy iteration or value iteration, one is merely provided with a policy that 
minimizes the expected cost (in our case time to completion), and no additional information is readily available.
With our approach instead, one 
not only obtains a policy minimizing the AVaR criterion, but a statistical description of the costs is also obtained as a byproduct. That is to say, for
each discretized completion time $k$, the probability $\Pr[c^{[d]}=k]$ is computed 
as well, thus unveiling the relationship between the time to complete
the deployment task and its probability.
This is shown in Figure \ref{fig:probcomparison} for different values of 
$\tau$.
 Hence, if the computed policy 
does not meet the desired performance, the designer has  information
on how to tune  $T$ and $P$.

\begin{figure}[tb]
\centering
\includegraphics[width=0.8\linewidth]{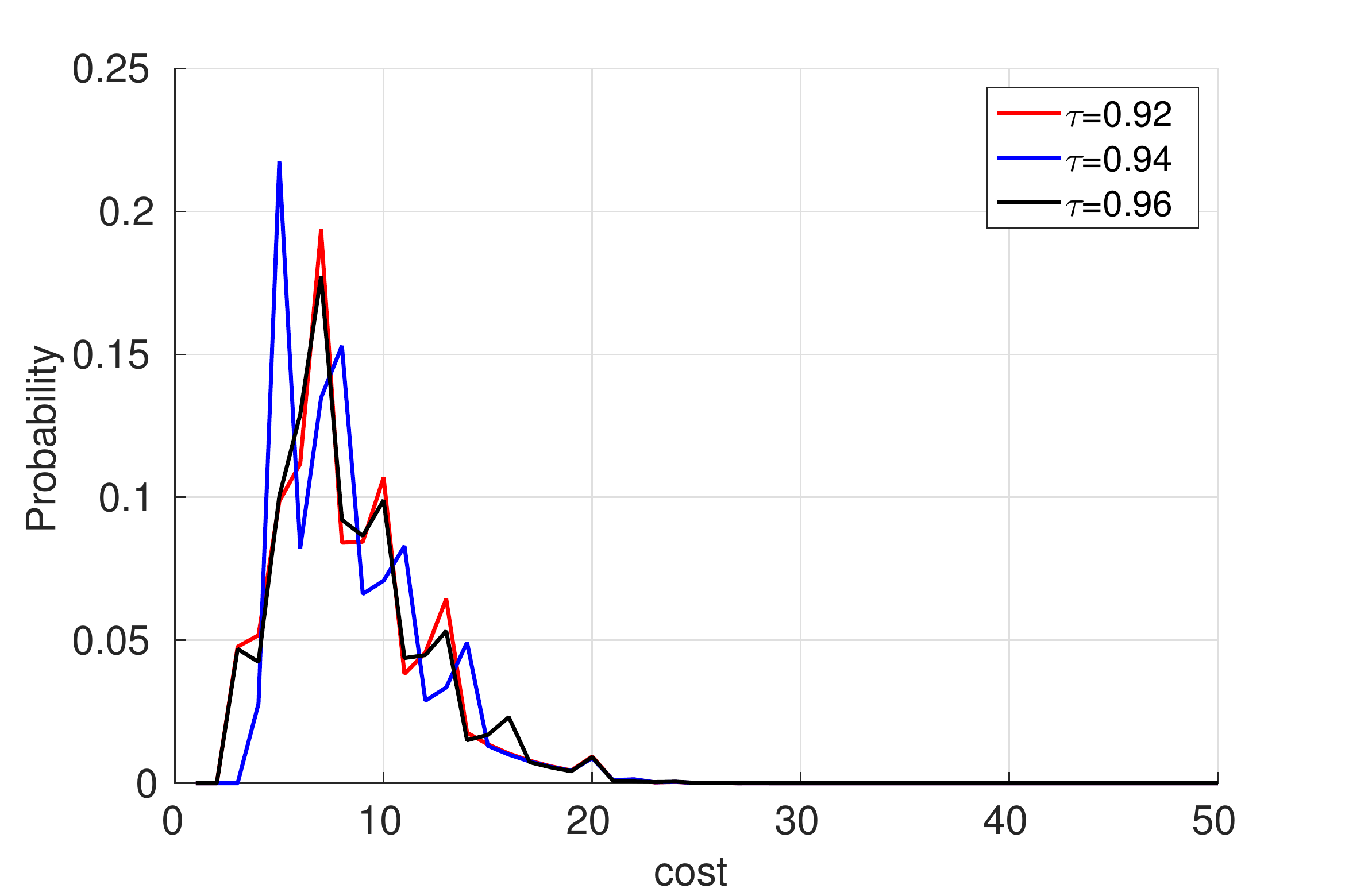}
\caption{Comparison of probability cost distribution for different values of $\tau$.}
\label{fig:probcomparison}
\end{figure}

\section{Conclusions}
\label{sec:conclusions}
In this paper we have considered how 
risk aversion in MDPs can be introduced jointly  with the AVaR risk metric under the total cost criterion.
Our results
advance the state of the art because AVaR has only been previously considered in MDPs
with finite horizon
or discounted infinite horizon cost criteria. Such extension is \emph{important} as the total cost criterion appears as a natural model for  robotic applications, and is \emph{non-straightforward} as current algorithms, e.g., from \cite{Jain2014} and \cite{Ott2011}, only work with bounded cumulated costs (which is not the case for total cost formulations). Under two mild assumptions, an approximation algorithm with provable sub-optimality gap was provided.
Furthermore, a rapid deployment scenario was used to demonstrate that
risk-aversion gives more informative policies when compared to
traditional risk-neutral formulations.
While our findings focus on risk averse MDPs with an AVaR risk metric, 
our approach can be easily extended along multiple dimensions. In particular, by exploiting the results presented in \cite{Jain2014}, it is
possible to use our approximation for a broader range of risk metrics, i.e.,
metrics that are uniformly continuous and law invariant. Moreover, 
since the algorithm we considered is based on occupancy measures,
it can be easily extended to the CMDP case. This will be the focus of future work.

\end{document}